\documentclass[11pt,a4paper,english]{amsart}
\usepackage{babel}
\usepackage[latin1]{inputenc}
\usepackage{amsmath}
\usepackage{amsthm}
\usepackage{amsfonts}
\usepackage{indentfirst}
\usepackage{graphicx}

\usepackage{amssymb}
\usepackage[mathscr]{eucal}

\newtheorem{de}{Definition}%[section]
\newtheorem{pro}{Proposition}%[section]
\newtheorem{teo}{Theorem}%[section]
\newtheorem{rem}{Remark}%[section]
\newtheorem{lem}{Lemma}%[section]
\newtheorem{exa}{Example}%[section]
\newtheorem{alg}{Algorithm}%[section]
\newcommand{\co}{{\mathcal O}}

\newcommand{\gp}{\mathbb{P}}

\newcommand{\gz}{\mathbb{Z}}
\newcommand{\gq}{\mathbb{Q}}

\newcommand{\ck}{\mathcal K}
\newcommand{\cb}{{\mathcal B}}

\newcommand{\cs}{{\mathcal S}}

\newcommand{\ct}{{\mathcal T}}
\newcommand{\cp}{{\mathcal P}}

\newcommand{\cf}{{\mathcal F}}
\newcommand{\calr}{{\mathcal R}}

\renewcommand{\int}{{\rm int}}

\newcommand{\pic}{{\rm Pic}}

\title[Poincaré Problem and dicritical divisors]{The Poincaré Problem, algebraic integrability and dicritical divisors}
\author{C.~Galindo \and F.~Monserrat}
\curraddr{\texttt{Carlos Galindo:} Instituto Universitario de Matem\'aticas y Aplicaciones de
Castell\'on and Departamento de Matem\'aticas, Universitat Jaume I,
Campus de Riu Sec. s/n, 12071 Castell\'{o} (Spain).} \email{
galindo@mat.uji.es} \curraddr{\texttt{Francisco Monserrat:} Instituto Universitario de
Matem\'atica Pura y Aplicada, Universidad Polit\'ecnica de Valencia,
Camino de Vera s/n, 46022 Valencia (Spain).}
\email{framonde@mat.upv.es}
\date{}
\thanks{Supported by Spain Ministry of Education
 MTM2007-64704 and Bancaixa P1-1B2009-03}
\date{}

\begin{document}

\maketitle

\begin{abstract}
We solve the Poincaré problem for plane foliations with only one
dicritical divisor. Moreover, in this case, we give an algorithm
that decides whether a foliation has a rational first integral and
computes it in the affirmative case. We also provide an algorithm to
compute a rational first integral of prefixed genus $g\neq 1$ of any
type of plane foliation $\cf$. When the number of dicritical
divisors dic$(\cf)$ is larger than two, this algorithm depends on
suitable families of invariant curves. When dic$(\cf) = 2$, it
proves that the degree of the rational first integral can be bounded
only in terms of $g$, the degree of $\cf$ and the local analytic
type of the dicritical singularities of $\cf$.
\end{abstract}

\section{Introduction and Results}

Denote by $\gp^2$ the projective plane over the field of complex
numbers.  Poincar\'e, in \cite{poi2}, observed that ``to find out
whether a differential equation of the first order and of the first
degree is algebraically integrable, it is enough to get an upper
bound on the degree of the integral. Afterwards, one only needs to
perform purely algebraic computations." The motivation for this
observation, expressed in modern terminology, was the problem of
deciding whether a singular algebraic foliation $\cf$ on $\gp^2$
(plane foliation) has a rational first integral and, when the answer
is positive, to compute it. The so-called Poincar\'e problem
consists of obtaining an upper bound of the degree of the first
integral depending only on the degree of the foliation. Although it
is well-known that such a bound does not exist in general, in the
forthcoming Theorem \ref{teor3} (our main goal) we shall give a
bound of this type under the assumption that the minimal resolution
of the singularities of $\cf$ (which exists by a result of
Seidenberg \cite{seid}) has only one dicritical (i.e., non-invariant
by $\cf$) exceptional divisor. Also, under this assumption, we shall
give an algorithm that solves the above mentioned decision problem
whose inputs are a differential 1-form $\mathbf{\Omega}$ defining
the foliation and the part of the minimal resolution of $\cf$
corresponding to dicritical singularities (see Definition
\ref{def1}), which can be obtained from $\mathbf{\Omega}$. This
algorithm only involves simple integer arithmetics and resolution of
systems of linear equations.

The natural extended version of the Poincaré problem consists of
bounding the degree of the algebraic integral (reduced and
irreducible) invariant curves of a foliation $\cf$ (without assuming
algebraic integrability) in terms of data obtained from the
foliation and/or invariants related with the invariant curves
themselves. There has been (and there is) a lot of activity
concerning this or related problems, some of the main results
(including higher dimension) being
\cite{ce-li,car,ca-ca,zam1,soa1,soa2,zam2,pere,es-kl,g-m-1,c-l,g-m-2}.

The above mentioned problem was stated at the end of the 19th
century as the problem of deciding whether a complex polynomial
differential equation on the complex plane is algebraically
integrable. The usefulness of nonlinear ordinary differential
equations in practically any science turns this problem into a very
attractive one, especially because when a differential equation
admits a first integral, its study can be reduced in one dimension
and because it is related with other interesting challenges. For
example, it is related with the second part of the XVI Hilbert
problem which tries to bound the number of limit cycles for a real
polynomial vector field \cite{lli2, lli}, with the solutions of
Einstein's field equations in general relativity \cite{hew} and with
the center problem for vector fields \cite{sch, d-l-a}.

Algebraic integrability problem has a long history. In the 19th
century, the main contributors  were Darboux \cite{dar}, Poincar\'e
\cite{poi1, poi2}, Painlev\'e \cite{pai} and Autonne \cite{aut}.
They laid the foundations of a theory that has inspired a large
quantity of papers, many of them published in the last twenty years.
It was Darboux who gave a bound on the number of invariant  integral
algebraic curves of a polynomial differential equation that, when it
is exceeded, implies the existence of a first integral. A close
result was proved by Jouanolou \cite{jou} to guarantee that a
foliation $\cf$ as above has a rational first integral and that if
one has enough reduced invariant curves, then the rational first
integral can be computed. The existence of a first integral of that
type is also equivalent to the fact that every invariant curve by
$\cf$ is algebraic and to the fact that there exist infinitely many
invariant integral  curves. These results have been adapted and
extended to foliations on other varieties \cite{jou2,jou,
b-n,ghy,cor}. In \cite{g-m-1}, the authors gave an algorithm to
decide about the existence of  a rational first integral (and to
compute it in the affirmative case) assuming that one has a
well-suited set of dic$(\cf)$ reduced invariant curves, where we
stand dic$(\cf)$ for the number of dicritical divisors appearing in
the resolution of $\cf$. In the same paper, it was also shown how to
get sets of invariant curves as above for foliations such that the
cone of curves of the surface obtained by the resolution of the dicritical
singularities is polyhedral.

Painlev\'e in \cite{pai} posed the problem of recognizing the genus of the general
invariant algebraic curve of a foliation admitting a rational first
integral. Mixing the ideas of Painlev\'e and Poincar\'e, one can try to bound
the degree of the rational first integral using also its genus. When
$\cf$ is non-degenerated, Poincar\'e himself provided a bound
proving that $d (r-4) \leq 4 (g-1)$, where $d$ (respectively, $r$)
is the degree of the first integral (respectively, $\cf$) and $g$
the mentioned genus. In the same sense, for foliations $\cf$ as
above with Kodaira dimension equal to $2$, there exists a bound on
the degree of  the rational first integral which only depends on its
genus, the degree of $\cf$ and the sequence
$\{$h$^0(\gp^2,\mathcal{K}_\cf^{\otimes m})\}_{m > 0}$, $\mathcal{K}_\cf$ being the
canonical sheaf of the foliation $\cf$ (see \cite{pere} for a
proof). With this philosophy, we shall show in clause a) of Theorem
\ref{teor1} that, for a foliation $\cf$ on $\gp^2$ having a rational first
integral of genus $g\not=1$ and such that dic$(\cf) \leq 2$, there exists a bound on the
degree of the first integral which only depends on $g$, the degree of $\cf$ and the local
analytic type of the dicritical singularities of $\cf$. It is worthwhile to add that, in \cite{l-n}, Lins Neto showed that, in general, such a bound does not exist. Clause b) of Theorem \ref{teor1} states that, for foliations $\cf$ of $\gp^2$ satisfying also dic$(\cf) \leq 2$, there exists an algorithm of the same type as the one of clause b) of Theorem \ref{teor3} that decides whether $\cf$ has a rational first integral of fixed genus $g\not=1$ (and computes it in the affirmative case).

Theorem \ref{teor2} extends the results of Theorem \ref{teor1} to
the case when dic$(\cf)\geq 3$. Here it is required, as an
additional hypothesis, the existence (and the knowledge) of a set of
dic$(\cf)-2$ independent algebraic solutions of $\cf$ (see
Definition \ref{ind}). In a sense, this theorem is related with the
above mentioned Darboux and Jouanolou's results because the
knowledge of enough invariant curves allows us to obtain information
concerning the rational first integral.

We finish this introduction stating the main results and summarizing briefly the aim of each section of the paper.

\begin{teo}\label{teor3}

Let $\cf$ be a singular algebraic foliation on $\gp^2$ of degree $r$ such that $ \mathrm{dic} (\cf)=1$.
\begin{description}
\item[a)] If $\cf$ is algebraically integrable and $d$ is the degree of a general integral invariant curve, then
$$d\leq \frac{(r+2)^2}{4}.$$
Therefore, the Poincaré problem is solved in this case.

\item[b)] There exists an algorithm to decide whether $\cf$ has a rational first integral (and to compute it, in the affirmative case) whose inputs are: an homogeneous $1$-form defining $\cf$ and the minimal resolution of the dicritical singularities of $\cf$.

\end{description}

\end{teo}

\begin{teo}\label{teor1}
Let $\cf$ be a singular algebraic foliation on $\gp^2$ such that $ \mathrm{dic} (\cf)\leq 2$. Let $g\not=1$ be a non-negative integer.
\begin{description}
\item[a)] Assume that $\cf$ has a rational first integral of genus $g$. Then, there exists a bound on the degree of the first integral depending only on the degree of $\cf$, a bound on $g$ and the local analytic type of the dicritical singularities of $\cf$.
  \item[b)] There exists an algorithm to decide whether $\cf$ has a rational first integral of genus $g$ (and to compute it, in the affirmative case) whose inputs are: a homogeneous $1$-form defining $\cf$ and the minimal resolution of the dicritical singularities of $\cf$.

\end{description}
\end{teo}

\begin{teo}\label{teor2}
Let $\cf$ be a singular algebraic foliation on $\gp^2$ such that $ \mathrm{dic} (\cf) \geq 3$ and assume the existence (and the knowledge) of a $[  \mathrm{dic} (\cf) -2]$-set $S$ of independent algebraic solutions of $\cf$ (see Definition \ref{ind}). Let $g\not=1$ be a non-negative integer.
\begin{description}
\item[a)] Assume that $\cf$ has a rational first integral of genus $g$. Then there exists a bound on the degree of the first integral which depends on the degree of $\cf$, a bound on $g$, the local analytic type of the dicritical singularities of $\cf$ and the degrees of the curves in $S$ and their multiplicities at the centers of the sequence of blow-ups $\pi_\cf$ giving rise to the minimal resolution of the dicritical singularities of $\cf$.

\item[b)] There exists an algorithm to decide whether $\cf$ has a rational first integral of genus $g$ (and to compute it, in the affirmative case). Its inputs are: a homogeneous $1$-form defining $\cf$, $\pi_\cf$ and the degrees of the curves in $S$ and their above mentioned multiplicities.

  \end{description}
\end{teo}

Section \ref{sec2} provides the notations and preliminary facts devoted to
make easier the reading of the paper. Section \ref{sec3} contains
the mentioned study of rational first integrals with fixed genus; we describe the algorithm announced in clause b) of Theorem \ref{teor2} (Algorithm \ref{alg1}), which is supported mainly in Lemma \ref{lema1}; the algorithm of clause b) of Theorem \ref{teor1} is nothing but a particular case. Clause a) in both theorems is deduced as a consequence of the obtained algorithm. Section \ref{sec4} is devoted to prove Theorem \ref{teor3}. We describe first the algorithm of clause b), which is divided in two parts (Algorithms \ref{alg2} and \ref{alg3}). The bound on the degree of the first integral is deduced from  the auxiliary results supporting the algorithm. Finally, in Section \ref{sec5}, we give several examples that show
how our algorithms work.

\section{Preliminaries}
\label{sec2}
\subsection{Basic definitions}

Let $Z$ be an algebraic smooth projective complex surface. A
\emph{singular algebraic foliation} $\mathcal F$ (or simply a
\emph{foliation} in the sequel) on $Z$ is given by a set of pairs
$\{(U_i,v_i)\}_{i \in I}$, where $\{U_i\}_{i \in I}$ is an open
covering of $Z$, $v_i\in {\mathcal T}_Z(U_i)$ (where ${\mathcal
T}_Z$ denotes the tangent sheaf of $Z$) and, if $i,j,k \in I$, there
exist functions $g_{ij}\in {\mathcal O}^*_{Z}(U_i\cap U_j)$ such
that $v_i=g_{ij}v_j$ on $U_i\cap U_j$ and $g_{ij}g_{jk}=g_{ik}$ on
$U_i\cap U_j\cap U_k$. If ${\mathcal L}$ denotes the invertible
sheaf defined by the multiplicative cocycle given by $\{g_{ij}\}$,
we can regard $\mathcal F$ as a global section of the sheaf
${\mathcal L}\otimes {\mathcal T}_Z$ and, therefore, there is an
induced morphism ${\mathcal L}^{-1}\rightarrow {\mathcal T}_Z$. The
sheaf ${\mathcal L}$ is called the \emph{canonical sheaf} of the
foliation $\cf$, and it will be denoted by ${\mathcal K}_{\cf}$.
Conversely, given an invertible sheaf ${\mathcal J}$ on $Z$ and a
morphism ${\mathcal J} \rightarrow {\mathcal T}_Z$, we can
canonically associate with ${\mathcal J}$ a foliation  such that
${\mathcal J}^{-1}$ is its canonical sheaf.

From the dual point of view, the natural product map
$\Omega_{Z}^1 \otimes \Omega_{Z}^1 \rightarrow \Omega_{Z}^2$ gives rise to an isomorphism $\Omega_{Z}^1\rightarrow {\mathcal T}_{Z} \otimes \Omega_{Z}^2$. Under this isomorphism the map $\ck_{\cf}^{-1}\rightarrow {\mathcal T}_{Z}$ corresponds to a global section of $\Omega_Z^1 \otimes \ck_{\cf} \otimes \ck_Z^{-1}$, where ${\mathcal K}_Z$ denotes the canonical sheaf of $Z$.

Given a point $p\in Z$, take an open set $U_i$ such that $p\in U_i$.
The \emph{algebraic multiplicity} of $\cf$ at $p$, $\nu_p(\cf)$, is
the order of $v_i$ at $p$, that is, $\nu_p(\cf)=s$ if and only
if $(v_i)_p\in {\rm m}_p^s {\mathcal T}_{Z,p}$ and
$(v_i)_p\not\in {\rm m}_p^{s+1} {\mathcal T}_{Z,p}$, where ${\rm
m}_p$ denotes the maximal ideal of the local ring $\co_{Z,p}$. The
\emph{singularities} of $\cf$ are those points $p$ in $Z$ such that
$\nu_p(\cf)\geq 1$. We shall assume that all considered foliations are
saturated, that is they have finitely many singularities. Notice that if
$\cf \in H^0 (Z,\ck_\cf \otimes {\mathcal T}_{Z})$ vanishes on a divisor $H$
of $Z$, one can regard $\cf$ as a global section of $\ck_\cf \otimes {\mathcal T}_{Z} \otimes {\mathcal O}_Z (-H)$
which defines a foliation $\cf^s$, called saturation of $\cf$, with isolated singularities such that
$\ck_{\cf^s} = \ck_\cf \otimes {\mathcal O}_Z (-H)$.

Recall that an integral (i.e., reduced and irreducible) projective
curve $C\subseteq Z$ is called to be invariant by $\cf$ if the restriction
map $\ck_{\cf}^{-1}\mid_C \rightarrow {\mathcal T}_{Z}\mid_{C}$
factorizes through the natural inclusion ${\mathcal
T}_{C}\rightarrow {\mathcal T}_{Z}\mid_{C}$ and that a projective
curve $C\subseteq Z$ is named \emph{invariant} by $\cf$ if all its
integral components are invariant. Integral invariant curves of a foliation $\cf$
are usually called {\it algebraic solutions} of $\cf$. Locally, it means that for all closed
point $p \in Z$, $v_p(f) \in I_{C,p}$, whenever $f \in I_{C,p}$, $I_{C,p}$ being the ideal of $C$
and $v_p$ a generator of $\cf$ both at $p$; or, dually, that the local differential 2-form $\omega_p \wedge
d f$ is a multiple of $f$, $\omega_p$ being a local equation as a form of $\cf$ and $f=0$ a local equation of
$C$ at $p$.

Assume now that $\cf$ is a foliation on  $\gp^2$ (the projective
plane over the complex field) and let $r$ be the
non-negative integer such that ${\mathcal
K}_{\cf}=\co_{\gp^2}(r-1)$; $r$ is named the \emph{degree} of the
foliation. The Euler sequence $ 0 \rightarrow \Omega^1_{\gp^2}
\rightarrow \mathcal{O}_{\gp^2}(-1)^3 \rightarrow
\mathcal{O}_{\gp^2} \rightarrow 0 $, in fact the dual sequence $ 0
\rightarrow \mathcal{O}_{\gp^2}  \rightarrow
\mathcal{O}_{\gp^2}(1)^3 \rightarrow {\mathcal T}_{\gp^2}
\rightarrow 0 $, allows us to regard $\cf$ as induced by a
homogeneous vector field
$$\mathbf{X} = U \partial / \partial X_0 + V
\partial / \partial X_1 + W \partial / \partial X_2,$$ where $U,V,W$ are
homogeneous polynomials of degree $r$ in homogeneous coordinates
$(X_0:X_1:X_2)$ on $\gp^2$; two vector fields define the same
foliation if, and only if, they differ by a multiple of the radial
vector field of the form $ H(X_0,X_1,X_2) (X_0 \partial / \partial
X_0 + X_1
\partial / \partial X_1 + X_2 \partial / \partial X_2)$, where $H$ is a homogeneous
polynomial of degree $r -1 $. A detailed description
of this fact, using coordinates,  can be seen in \cite[Cap\'{\i}tulo 1.3]{G-M}.

Returning to the dual point of view, the foliation $\cf$ corresponds
to a global section of the sheaf $\Omega_{\gp^2}^1\otimes
\co_{\gp^2}(r+2)$. Taking into account the Euler sequence, this
section corresponds to three homogeneous polynomials $A$, $B$ and
$C$ of degree $r+1$, without common factors, such that
$X_0A+X_1B+X_2C=0$ (Euler condition); equivalently, the section can
be seen as the homogeneous differential 1-form on ${\mathbb A}^3$:
$$\mathbf{\Omega}:=AdX_0+ BdX_1+ CdX_2.$$ Notice that the equality
\[
\det \left(
       \begin{array}{ccc}
         d X_0 & d X_1 & d X_2 \\
         X_0 & X_1 & X_2 \\
         U & V & W \\
       \end{array}
     \right) = \mathbf{\Omega}
\]
allows us to compute $\mathbf{\Omega}$ from $\mathbf{X}$ and that a curve on $\gp^2$ defined by a homogeneous equation
$F=0$  is invariant by $\cf$ if, and only if, the polynomial $F$ divides the projective 2-form $\mathbf{\Omega} \wedge d F$.

\subsection{Resolution of singularities}
\label{resolution}

Throughout this paper, we shall consider sequences of morphisms
\begin{equation}
\label{seq} X_{n+1} \mathop  {\longrightarrow} \limits^{\pi _{n} }
X_{n} \mathop {\longrightarrow} \limits^{\pi _{n-1} }  \cdots
\mathop {\longrightarrow} \limits^{\pi _2 } X_2 \mathop
{\longrightarrow} \limits^{\pi _1 } X_1 : = \gp^2,
\end{equation}
such that each $\pi_i$ is the blow-up of the variety $X_i$ at a closed point $p_i\in
X_i$, $1\leq i\leq n$. The  set of closed points $ \mathfrak{K} =
\{p_1,p_2,\ldots,p_n\}$ given by a sequence as (\ref{seq}) will be called a {\it configuration} over
$\gp^2$ and the variety $X_{n+1}$ the {\it sky} of $\mathfrak{K}$; we
shall identify two configurations with $\gp^2$-isomorphic skies. We shall
denote by $E_{p_i}$ (respectively, $\tilde{E}_{p_i}$,
$E_{p_i}^*$) the exceptional divisor provided by the blow-up
$\pi_i$ (respectively, its strict transform, its total transform on $X_{n+1}$). Also, given two
points $p_i, p_j$ in $\mathfrak{K}$, we shall say that $p_i$ is {\it
infinitely near to} $p_j$ (denoted $p_i \geq p_j$) if either
$p_i=p_j$ or $i>j$ and $\pi_j\circ \pi_{j+1}\circ \cdots \circ
\pi_{i-1}(p_i)=p_j$. The relation $\geq$ is a partial ordering among
the points of the configuration $\mathfrak{K}$. Furthermore, a
point $p_i$ will be called {\it proximate} to other one $p_j$ whenever $p_i$ is
in the strict transform of the exceptional divisor $E_{p_j}$ on the surface containing $p_i$.
To represent sequences as (\ref{seq}), we shall use a combinatorial invariant named the {\it proximity graph}.
It is a graph whose vertices correspond to to the points $p_i$ in  $\mathfrak{K}$ and the edges join points
$p_i$ and $p_j$ whenever $p_i$ is proximate to $p_j$. This edge is dotted excepting the case when $p_i \in E_{p_j}$.

If $\cf$ is a foliation on $\gp^2$, a sequence of morphisms
(\ref{seq}) induces, for each $i=2,3,\ldots,n+1$, a foliation
$\cf_i$ on $X_i$ given by the pull-back of $\cf$ (see \cite{brun},
for instance). By a result of Seidenberg \cite{seid} there exists a
\emph{resolution of singularities} of $\cf$, that is, a sequence of
blow-ups as  (\ref{seq}) such that the foliation $\cf_{n+1}$ on
the last obtained surface $X_{n+1}$ has only simple singularities.
A singularity $p \in U_i$ is  \emph{simple} (or \emph{reduced}) if at least
one of the eigenvalues $\alpha$ and  $\beta$ of the linear part of the vector field
$v_i$ (that are well defined since $v_i(p)=0$) does not vanish and, assuming $\beta \neq 0$,
the quotient $\alpha/\beta$ is not an strictly positive rational number. These singularities have the
property that they cannot be removed by blowing-up.

In the sequel, we shall denote by $\cs_{\cf}$ the
configuration $\{p_i\}_{i=1}^n$ given by the centers of the blow-ups
involved in a \emph{minimal} (with respect to the number of
blow-ups) resolution of singularities of $\cf$, however in our development
we shall not use the whole resolution of singularities, but only the
sequence of blow-ups concerning the so-called configuration of \emph{dicritical} points
that we define next.

\begin{de}\label{def1}

{\rm

 An exceptional divisor $E_{p_i}$ (respectively, a point $p_i\in
\cs_{\cf}$) of a minimal resolution of singularities of a foliation
$\cf$ on $\gp^2$ is called {\it non-dicritical} if it is invariant
by the foliation $\cf_{i+1}$ (respectively, all the exceptional
divisors $E_{p_j}$, with $p_j\geq p_i$, are non-dicritical).
Otherwise, $E_{p_i}$ (respectively,  $p_i$) is said to be {\it
dicritical}. We shall denote by ${\mathcal B}_{\mathcal F}$ the
configuration of dicritical points in $\cs_{\mathcal F}$ and by
$Z_{\cf}$ the sky of $\cb_{\cf}$.

}
\end{de}

\subsection{Foliations having a rational first integral}
This paper is devoted to study algebraic integrability of certain type of foliations on the projective plane $\gp^2$, so we start this brief section by defining this concept.

\begin{de}

{\rm
A \emph{rational first integral} of a foliation $\cf$ on $\gp^2$ is
a rational map $f:\gp^2\cdots \rightarrow \gp^1$ such that the
closures of its fibers are invariant curves by $\cf$. Equivalently, and from an algebraic
point of view, if $f$ is given by a rational function $R$, $f$ is a rational first integral
if, and only if, $\mathbf{\Omega} \wedge d R =0$.  $\cf$ is called to be \emph{algebraically integrable} (or that it has a \emph{rational first integral}) whenever there exists such a rational map.

}
\end{de}

Consider an algebraically integrable foliation $\cf$ on $\gp^2$. The
second theorem of Bertini \cite{kle} shows that $\cf$ admits a \emph{primitive}
rational first integral $f:\gp^2\cdots \rightarrow \gp^1$ (that is,
such that the closures of its general fibers are integral curves).
Taking projective coordinates, if $F(X_0,X_1,X_2)$ and $G(X_0,X_1,X_2)$ are the
two homogeneous polynomials of the same degree $d$ which are the
components of $f$, then the closures of the fibers of $f$ are the
elements of the irreducible pencil $\cp_{\cf}:=\left\langle {F,G}
\right\rangle\subseteq H^0(\gp^2,\co_{\gp^2}(d))$ and, moreover, any
algebraic solution of $\cf$ is a component of a curve in
$\cp_{\cf}$. The \emph{degree} (respectively, \emph{genus}) of the
first integral $f$ will be the degree (respectively, geometric genus) of a
general fiber of $\cp_{\cf}$.

Let $Z_{\cf}$ be the sky of the configuration $\cb_{\cf}$ (see
Definition \ref{def1}). By comparing both processes, that of
elimination of indeterminacies of the rational map $f$ \cite[Theorem
II.7]{beau} and the resolution of the dicritical singularities of
$\cf$ through $\pi_{\cf}: Z_{\cf}\rightarrow \gp^2$, it can be
proved that $\pi_{\cf}$ is also the minimal resolution of the
indeterminacies of $f$. Indeed, if $p_i \in \mathfrak{K}$ and
$\mathfrak{f}$ and $\mathfrak{g}$ are local equations at $p_i$ of
the strict transform of two general elements $F$ and $G$ generating
the pencil $\cp_{\cf}$, then the local solutions of the foliations
$\cf_i$ at $p_i$ are the irreducible components of the local pencil
in the completion with respect to the maximal ideal of ${\mathcal
O}_{X_i,p_i}$ generated by $\mathfrak{f}$ and $\mathfrak{g}$ (see
\cite{julio} for complete details). As a consequence if $f$ is a
rational first integral of a foliation $\cf$ as above, then the map
$\tilde{f}:=f\circ \pi_{\cf}:Z_{\cf}\rightarrow \gp^1$ is a morphism.

\section{Rational first integral with given genus}
\label{sec3}

Along this section $\cf$ will be a foliation on $\gp^2$ of degree $r$ and $Z_\cf$ the sky of its
configuration ${\mathcal B}_{\mathcal F}$ of dicritical points.
Denote by $A(Z_{\cf})$ the vector space over $\gq$,
$\pic(Z_{\cf})\otimes_{\gz} \gq$, where $\pic(Z_{\cf})$ stands for
the Picard group of the surface $Z_{\cf}$. Intersection theory provides
a $\gz$-bilinear form:
$\pic(Z_{\cf})\times \pic(Z_{\cf})\rightarrow \gz$
which induces a non-degenerate bilinear form over $\gq$:
$ A(Z_{\cf}) \times A(Z_{\cf}) \rightarrow \gq$. The image by this form of a pair
$(x,y)\in A(Z_{\cf})\times A(Z_{\cf})$ will be denoted $x\cdot y$.

Given a divisor $A$ on $Z_{\cf}$, we shall denote by $[A]$ its class
in the Picard group $\pic(Z_{\cf})$ and also its image into
$A(Z_{\cf})$. If $C$ is either a curve on $\gp^2$ or an exceptional
divisor, $\tilde{C}$ (respectively, $C^*$) will denote its strict
(respectively, total) transform on the surface $Z_{\cf}$ via the
composition of blow-ups $\pi_{\cf}$. It is well known that the set
$\mathbf{B}:=\{[L^*]\}\cup \{ [E_{p}^*]\}_{p\in {\mathcal B}_{\cf}}$
is a $\gz$-basis (respectively, $\gq$-basis) of $\pic(Z_{\cf})$
(respectively, $A(Z_{\cf})$), where $L$ denotes a general line on
$\gp^2$.

Now, let us suppose that $\cf$ admits a rational first integral $f$ (which we assume to be primitive) and set
$\tilde{\cf}$ the foliation on $Z_{\cf}$ given by the
pull-back of $\cf$ by $\pi_{\cf}$. $\tilde{f}:=f\circ \pi_{\cf}$
is a first integral of $\tilde{\cf}$ and the integral
invariant curves of $\tilde{\cf}$ (which coincide with the integral components of the fibers of
$\tilde{f}$) are, on the one hand, the strict transforms on
$Z_{\cf}$ of the integral invariant curves of $\cf$ and, on the other hand,
the strict transforms of the exceptional divisors $E_{p_i}$ (with
$p_i\in \cb_{\cf}$) which are non-dicritical. Denote by $D_{\tilde{f}}$ a general fiber of $\tilde{f}$.
% and by $
%\mathcal{L}$ the ideal sheaf locally defined by the equations of two general elements of the pencil $\cp_\cf$, which we assume is of degree $d$. Cayley-Bacharach theorem \cite{eis} (applied to the sheaves $\mathcal{L}(d-3)$ and $\overline{\mathcal{L}}(d)$, $\overline{\mathcal{L}}$ being the integral closure of $\mathcal{L}$) and the fact that the sheaves $\pi_{\cf_*} (\mathcal{O}_{Z_\cf}(D_{\tilde{f}}))$ and $\overline{\mathcal{L}}(d)$ coincide allows us to prove that $h^1(Z_\cf, \mathcal{O}_{Z_\cf}(D_{\tilde{f}}))$ equals the genus of $f$ (see the proof of \cite[Lemma 1]{g-m-1}). Since $D_{\tilde{f}}^2 =0$ by B\'ezout theorem, Adjunction formula and Riemann-Roch theorem for $D_{\tilde{f}}$ prove that

It is clear that every curve defined by an element of the pencil $\cp_{\cf}$ is the push-forward by $\pi_{\cf}$
of some curve in $Z_{\cf}$ that is linearly equivalent to $D_{\tilde{f}}$; so there is an inclusion of $\cp_{\cf}$ into the space of global sections $H^0(\gp^2, {\pi_{\cf}}_*\mathcal{O}_{Z_\cf}(D_{\tilde{f}}))$. This inclusion is, in fact, an equality. Indeed, reasoning by contradiction, if we assume that the inclusion is strict, it follows that there exist infinitely many linearly equivalent to $D_{\tilde{f}}$ curves $C$ on $Z_{\cf}$ whose push-forwards to $\gp^2$ are not defined by elements in the pencil $\cp_{\cf}$. But, taking into account that $D_{\tilde{f}}^2=0$, all these curves $C$ satisfy $D_{\tilde{f}}\cdot C=0$ and, therefore, they must be contracted by $\tilde{f}$; so their push-forwards to $\gp^2$ are reducible curves whose integral components are also components of curves defined by elements in $\cp_{\cf}$. This is a contradiction because the pencil $\cp_{\cf}$ is irreducible.
%Proposition 3.4 of \cite{saopaulo} shows that $h^0(Z_\cf, \mathcal{O}_{Z_\cf}(D_{\tilde{f}})) =2$ and, as a consequence, the complete linear system $| D_{\tilde{f}}|$ induces $\tilde{f}$ and, through $\pi$, a one to one correspondence with the projectivisation of $\cp_{\cf}$.

Notice that the invariant by $\tilde{\cf}$ curves $C$ are those that satisfy $D_{\tilde{f}} \cdot C =0$ (because their irreducible components must be contracted by $\tilde{f}$). This condition and \cite[Corollary 1.21]{kollar} imply that
the class $[C]$ of $C$ in $A(Z_{\cf})$ belongs to the  boundary of the cone of curves of the surface $Z_\cf$ and so $C^2 \leq 0$.

We summarize the above ideas in the following result:

\begin{pro}\label{proposition1}

Let $\cf$ and $f$ be as above and let $D_{\tilde{f}}$ be a general fiber of $\tilde{f}$. Then,

\begin{itemize}

\item[(a)] $h^0(Z_{\cf},\co_{Z_{\cf}}(D_{\tilde{f}})): = \dim \left ( H^0(Z_{\cf},\co_{Z_{\cf}}(D_{\tilde{f}})) \right)=2$.

\item[(b)] A curve $C$ on $Z_\cf$ is invariant by $\tilde{\cf}$ if and only if
$D_{\tilde{f}}\cdot C=0$.

\item[(c)] If $C$ is a curve on $Z_{\cf}$ which is invariant by
$\tilde{\cf}$ then $C^2\leq 0$.

\end{itemize}

\end{pro}

\begin{rem}
{\rm  The equality $h^0(Z_\cf, \mathcal{O}_{Z_\cf}(D_{\tilde{f}})) =2$ proves that, to compute a
primitive rational first integral $f$ of $\cf$, it is enough to know
a divisor $T$ linearly equivalent to the strict transform on
$Z_{\cf}$ of a general fiber of the pencil $\cp_\cf$ and two
linearly independent global sections of ${\pi_{\cf}}_*\co_{Z_{\cf}}(T)$, which
will be the components of $f$.
}
\end{rem}

The morphism $\tilde{f}:Z_{\cf}\rightarrow \gp^1$ is a
fibration of the surface $Z_\cf$ by the curve $\gp^1$ in the sense
that $\tilde{f}$  is surjective and with connected fibers. Taking
duals as $\mathcal{O}_{Z_\cf}$-modules in the corresponding to $\tilde{f}$ sequence of differentials
$$0\rightarrow \tilde{f}^*\Omega_{\gp^1}^1\rightarrow
\Omega_{Z_{\cf}}^1\rightarrow \Omega_{Z_{\cf}/\gp^1}\rightarrow 0,$$
one gets
$$0\rightarrow \ct_{Z_{\cf}/\gp^1}\rightarrow
\ct_{Z_{\cf}}\rightarrow \tilde{f^*}\ct_{\gp^1},$$
where
$\ct_{\gp^1}$ and $\ct_{Z_{\cf}}$ denote the tangent sheaves of $\gp^1$ and $Z_{\cf}$, and
$\ct_{Z_{\cf}/\gp^1}$  the relative tangent sheaf of the fibration, which is an
invertible sheaf \cite[Section 1]{serrano}. The morphism
$\ct_{Z_{\cf}/\gp^1}\rightarrow \ct_{Z_{\cf}}$ is given by the
differential of $\tilde{f}$ and it defines the foliation
$\tilde{\cf}$, therefore we obtain the equality
$\ck_{\tilde{\cf}}=\ct_{Z_{\cf}/\gp^1}^{-1}$. From \cite[Lemma
1.1]{serrano}, it follows  that

$$\ck_{\tilde{\cf}}=\ct_{Z_{\cf}/\gp^1}^{-1}=\ck_{Z_{\cf}}\otimes
\tilde{f}^*\ck_{\gp^1}^{-1}\otimes \co_{Z_{\cf}} \left(-\sum_{i\in I}
(n_i-1)G_i \right),
$$
 where $\ck_{Z_{\cf}}$ and $\ck_{\gp^1}$ are the
canonical sheaves of $Z_{\cf}$ and $\gp^1$, respectively, and
$\{G_i\}_{i\in I}$  the set of integral components of
the singular fibers of $\tilde{f}$, $n_i$ being the multiplicity
of $G_i$ in the fiber to which belongs to. If we take
divisors $K_{\tilde{\cf}}$ and $K_{Z_{\cf}}$ such that
$\ck_{\tilde{\cf}}=\co_{Z_{\cf}}(K_{\tilde{\cf}})$ and
$\ck_{Z_{\cf}}=\co_{Z_{\cf}}(K_{Z_{\cf}})$, the above equality may
be rewritten in the following form
\begin{equation}\label{equality}
K_{\tilde{\cf}}-K_{Z_{\cf}}\sim 2D_{\tilde{f}}-\sum_{i\in I} (n_i-1)G_i,
\end{equation}
where $\sim$ means linear equivalence.

The
linear equivalence class of the divisor  $K_{\tilde{\cf}}-K_{Z_{\cf}}$ can also be expressed in terms of the above basis $ \mathbf{B}$ of $\pic(Z_{\cf})$. Indeed, denote, as above, the degree of $\cf$  by $r$ and by $\nu_p(\cf)$  the
algebraic multiplicity at $p$ of the foliation given by the
pull-back of $\cf$ on the surface to which $p$ belongs. Set
$\epsilon_p(\cf)$ the value $0$ (respectively, $1$) whenever the exceptional
divisor $E_p$ is non-dicritical (respectively, dicritical). Then, by \cite[Proposition 1.1]{cam}, it happens that
\[
\pi^* K_{\cf} - K_{Z_{\cf}} \sim \sum_{p\in \cb_{\cf}}
\left(\nu_p(\cf)+\epsilon_p(\cf) -1 \right)E_p^*,
\]
which gives the following equivalence

\begin{equation}\label{equality2}
K_{\tilde{\cf}}-K_{Z_{\cf}}\sim (r+2)L^*-\sum_{p\in \cb_{\cf}}
(\nu_p(\cf)+\epsilon_p(\cf))E_p^*.
\end{equation}

Next, we provide the concepts and results that will allow us to
state the algorithms that prove theorems \ref{teor3} to \ref{teor2}
in this paper. For a while, we shall assume that the foliation $\cf$
need not to have a rational first integral. Stand dic$(\cf)$ for the
number of dicritical exceptional divisors appearing in the minimal
resolution of $\cf$. The existence of a set of invariant curves for
$\cf$ as we are going to define is an hypothesis in Theorem
\ref{teor2}.

\begin{de}
\label{ind} {\rm Let $\cf$ be a foliation on $\gp^2$ and suppose
that $s:=$ dic$(\cf) \geq 3$. A [dic$(\cf) -2$]-{\it set of
independent algebraic solutions of} $\cf$ is a set
$S=\{C_1,C_2,\ldots,C_{s-2}\}$ of $s-2$  integral projective curves
on $\gp^2$, invariant by $\cf$, and such that the family of classes
in $A(Z_\cf)$
$$V(S):=\left\{[\tilde{C}_1],[\tilde{C}_2],\ldots,[\tilde{C}_{s-2}],
[K_{\tilde{\cf}}-K_{Z_{\cf}}] \right\}\cup \left\{ [\tilde{E}_p]\mid
p\in \cb_{\cf} \mbox{ and } {E}_p \mbox{ is non-dicritical}
\right\}$$ is $\gq$-linearly independent.

}
\end{de}

Let us consider the projective space over the field $\mathbb{Q}$
associated with the $\mathbb{Q}$-vector space $A(Z_{\cf})$:
$$\mathbb{P} A(Z_{\cf}):=(A(Z_{\cf})\setminus \{0\})/\mathbb{Q},$$
and denote $\mathbb{Q}x$ the element in $\mathbb{P} A(Z_{\cf})$
defined by the class $x \in A(Z_{\cf})$. The {\it primitive lattice
representative} of an element $\mathbb{Q}x$ in $\mathbb{P}
A(Z_{\cf})$ is the class in $A(Z_{\cf})$ of a divisor
$a_0L^*-\sum_{i=1}^n a_i E_{p_i}^*$ included in $\mathbb{Q}x$
satisfying $a_0>0$, $a_i\in \gz$ for all $i$ and
$\gcd(a_0,a_1,\ldots,a_n)=1$.

Sets $S$ as in Definition \ref{ind} determine the following subsets
of $\mathbb{P} A(Z_\cf)$, which will be useful in our algorithms:
$$\calr_{\cf}(S):=\{\mathbb{Q}x\in \mathbb{P} A(Z_{\cf})\mid x^2=0 \mbox{ and } z\cdot
x=0 \mbox{ for all } z\in V(S)\}.$$

{\it When} dic$(\cf) \leq 2$, {\it we shall say that $S=\emptyset$
is a} [dic$(\cf) -2$]-{\it set of independent algebraic solutions of
$\cf$} and $\calr_{\cf}(\emptyset)$ is defined as above,
$V(\emptyset)$ being the set of classes $$\left\{
[K_{\tilde{\cf}}-K_{Z_{\cf}}] \right\}\cup \left\{ [\tilde{E}_p]\mid
p\in \cb_{\cf} \mbox{ and } {E}_p \mbox{ is non-dicritical}
\right\}.$$

Assume again that $\cf$ has a (primitive) rational first integral $f$, let $D$ be a general element of $\cp_{\cf}$ and suppose that
$\pi_\cf$ is the composition of a sequence as in (\ref{seq}). Notice that $[\tilde{D}]=[D_{\tilde{f}}]$, where $D_{\tilde{f}}$ is as in the beginning of this section. Set
$$T_{\cf}:=s_0L^*-\sum_{i=1}^n s_i E_{p_i}^*$$ such that
$[T_{\cf}]$ is the primitive lattice representative of
$\mathbb{Q}[\tilde{D}]$; Clearly, this implies that
$[\tilde{D}]=\gamma [T_{\cf}]$ for some positive integer $\gamma$.
Then, we can state the following

\begin{lem}\label{lema1}
Let $\cf$ be a foliation on $\gp^2$ having a rational first integral such that
it admits a $[${\rm dic}$(\cf) -2]$-set of independent algebraic
solutions $S$. Then $\mathbb{Q}[T_{\cf}]\in \calr_{\cf}(S)$ and the
cardinality of $\calr_{\cf}(S)$ is either $1$ or $2$.

\end{lem}

\begin{proof}

The point $\mathbb{Q} [T_{\cf}]\in \mathbb{P}A_{Z_{\cf}}$ belongs to $\calr_{\cf}(S)$ as a consequence of the equivalence (\ref{equality}) and
clause (b) of Proposition \ref{proposition1}. Let us prove the second assertion.

 Set
$\cb_{\cf}=\{p_1,p_2,\ldots,p_n\}$ and take projective coordinates
$(X_0:X_1:\cdots:X_n)$ of $\mathbb{P} A(Z_{\cf})$ with respect to
the basis $\mathbf{B}$ of $A_{Z_{\cf}}$. Notice that
$$\calr_{\cf}(S)=\mathcal{Q} \cap \mathbb{P}\langle V(S)\rangle
^\bot,$$ where $\mathcal{Q}$ is the quadric defined by the equation
$X_0^2-\sum_{i=1}^n X_i^2=0$ and $\mathbb{P}\langle V(S)\rangle
^\bot$ the projective subspace of $\mathbb{P} A(Z_{\cf})$ associated
with the orthogonal subspace (with respect to the bilinear form
$\cdot$ defined at the beginning of this section) of the linear
subspace $\langle V(S)\rangle$ of $A(Z_{\cf})$ spanned by $V(S)$.

Let $\mathbb{Q} q$ be a point in $\mathcal{Q} \cap \mathbb{P}\langle V(S)\rangle ^\bot $. The
polar hyperplane $H_q:=\{\mathbb{Q} x \in \mathbb{P}A(Z_{\cf})\mid q\cdot x=0 \}$ of the point $\mathbb{Q} q$ with respect to the quadric $\mathcal{Q}$ is defined
by the equation $$Q_0X_0-\sum_{i=1}^n Q_iX_i=0,$$
where $(Q_0:Q_1:\cdots:Q_n)$ are projective coordinates of $\mathbb{Q} q$.
On the one hand, the hyperplane $H_q$ is tangent to the quadric since $\mathbb{Q} q\in {\mathcal Q}$ and, on the other hand, it
contains $\mathbb{P}\langle V(S)\rangle$ because $(L^*)^2=1$,
$L^*\cdot E_{p_i}^*=0$, $(E_{p_i}^*)^2=-1$ for all $i$ and
$E_{p_i}^*\cdot E_{p_j}^*=0$ if $i\not=j$, and because $Q_0[L^*]-\sum_{i=1}^n
Q_i[E_{p_i}^*]$ is orthogonal to $V(S)$.

Let $t \leq n$ be a nonnegative integer. Denote by
$Gr_t(\mathbb{P}(A(Z_{\cf})))$ the set of $t$-dimensional projective
linear subspaces of $\mathbb{P}A(Z_{\cf})$, that is, the
$t$-Grassmannian variety of  $\gp A(Z_{\cf})$. The polarity map
(with respect to the quadric $\mathcal Q$) $\mathbb{Q}p \mapsto H_p$
extends to polarity maps on the Grassmannians
$$G_t: Gr_t(\mathbb{P}A(Z_{\cf}))\rightarrow Gr_{n-1-t}(\mathbb{P}A(Z_{\cf})).$$

Assume that dic$(\cf) = s>1$. Then the dimension of
$\mathbb{P}\langle V(S)\rangle$ is $n-2$ and hyperplanes $H_q$ as
above (tangent to $\mathcal{Q}$ and containing $\mathbb{P}\langle
V(S)\rangle$) are the tangent hyperplanes at the points of
intersection between the quadric and the polar variety
$W:=G_{n-2}(\mathbb{P}\langle V(S)\rangle)$. $W$ is a projective
line  because it belongs to $Gr_{n-1-(n-2)}(\mathbb{P}A(Z_{\cf}))$
and, therefore, the number of intersection points of $
\mathbb{P}\langle V(S)\rangle ^\bot$ with the quadric $\mathcal{Q}$
must be less than or equal to 2. Since $H_q$ determines
$\mathbb{Q}q$ there are, at most, 2 possibilities for such points
$\mathbb{Q}q$.

Otherwise $s=1$ and then $\mathbb{P}\langle V(\emptyset)\rangle$ has
dimension $n-1$, so its polar variety $$W:=G_{n-1}(\mathbb{P}\langle
V(\emptyset)\rangle)$$ is a point $\mathbb{Q} q$  of the quadric
(which corresponds to $\gp \langle V(\emptyset)\rangle^\bot$) and
therefore $H_q=\mathbb{P}\langle V(\emptyset)\rangle$.
\end{proof}

\begin{rem}
\label{nota2} {\rm The above proof shows how to compute
$\calr_{\cf}(S)$ (in terms of the basis $\mathbf{B}$). In the case
when $\mathrm{dic}(\cf)\leq 2$, one can perform it only using data
obtained from the resolution of singularities of $\cf$. Indeed, with
these data one is able to compute the class
$[K_{\tilde{\cf}}-K_{Z_{\cf}}]$ (see  formula (\ref{equality2}))
and,  for each $p\in \cb_{\cf}$, one has
$[\tilde{E}_p]=[E^*_p]-\sum_q [E^*_q]$, where $q$ runs over the set
of points of $\cb_{\cf}$ which are proximate to $p$. When
$\mathrm{dic}(\cf)> 2$ one also needs the coordinates in the basis
$\mathbf{B}$ of the set of classes of strict transforms on $Z_{\cf}$
of the invariant by $\cf$ curves in $S$,
$\{[\tilde{C}_i]\}_{i=1}^{\mathrm{dic}(\cf) - 2}$.

%From these data we can compute $\calr_{\cf}(S)$ easily because it is the set $\mathcal{Q} \cap W$ (see the above proof).

}
\end{rem}

\begin{rem}\label{trajesparatodos}
{\rm If we do not assume that a foliation $\cf$ has a rational first
integral, then it also holds that the cardinality of the set
$\calr_{\cf}(S)$ is less than or equal to 2; moreover, in this case,
$\calr_{\cf}(S)$ may be empty. In addition, if
$\mathrm{dic}(\cf)=1$, $\calr_{\cf}(S)$ is either empty or its
unique element is  $\gp \langle V(\emptyset) \rangle^\bot$. These
facts are straightforward from the proof of Lemma \ref{lema1}.

 }
\end{rem}

The following algorithm is also a {\it proof of clause b) of
Theorems \ref{teor1} and \ref{teor2}.} It can be applied to
foliations $\cf$ admitting a [dic$(\cf) -2$]-set of independent
algebraic solutions and it decides about existence of a rational
first integral of $\cf$ of a prefixed genus $g\not=1$ (computing it
in the affirmative case).
%We have divided the result into two theorems to emphasize that,
%when dic$(\cf) \leq 2$, only the resolution of dicritical singularities
%is needed to run the algorithm.
The algorithm works because Proposition
\ref{proposition1} and Lemma \ref{lema1} hold and, when $\cf$ admits a
rational first integral, the divisor $D_{\tilde{f}}$ in Proposition
\ref{proposition1} must satisfy the Adjunction formula.\\

\begin{alg}\label{alg1}
$\;$   \newline \newline \noindent {\it Input:} {\rm A projective
differential 1-form $\mathbf{\Omega}$ defining $\cf$, a non-negative integer
$g\not=1$, the configuration $\mathcal{B}_{\cf}$ and  a
[dic$(\cf) -2$]-set $S=\{C_1,C_2,\ldots,C_{s-2}\}$ of independent algebraic solutions of $\cf$.\\

\noindent {\it Output:} Either a primitive rational first integral
of $\cf$ of genus $g$ or ``0'' (which means that such a first integral does not exist).\\

\begin{itemize}

\item[1.] If dic$(\cf) \leq 2$, define $S:=\emptyset$.

\item[2.] Compute the set $\calr_{\cf}(S)$. If
$\calr_{\cf}(S)=\emptyset$ then return ``0''. Else, let ${\mathcal
L}:=\calr_{\cf}(S)$.

\item[3.] While ${\mathcal L}\not= \emptyset$:

\begin{itemize}

\item[3.1.] Choose $\ell=\mathbb{Q}q\in {\mathcal L}$ and let ${\mathcal L}:={\mathcal L}\setminus \{\ell\}$.

\item[3.2.] Compute coordinates in the basis $\mathbf{B}$, $(d,-m_1,-m_2,\ldots,-m_n)$, of the primitive lattice representative  of $\ell$. If  $m_i<0$ for some $i$, then
go to Step 3.

\item[3.3.] Compute $$\alpha:= \frac{2(g-1)}{-3d + \sum_{i=1}^n m_i}.$$
If $\alpha$ is not  a positive integer then go to Step
3.

\item[3.4.] Compute the space of global sections $$H^0(\gp^2,{\pi_{\cf}}_* \co_{Z_{\cf}}(\alpha T))\subseteq H^0(\gp^2,\co_{\gp^2}(\alpha d)),$$
where $T:=dL^*-\sum_{i=1}^n m_iE_i^*$.

\item[3.5.] If the dimension of the above space is not 2, then go
to Step 3. Else, choose two homogeneous polynomials $F$ and $G$ of
degree $\alpha d$ generating that space.

\item[3.6.] If $\mathbf{\Omega} \wedge (GdF-FdG)=0$ then the rational map
$\gp^2\cdots \rightarrow \gp^1$ whose components are $F$ and $G$ is
a primitive rational first integral of $\cf$; return it. Else, go to
step 3.

\end{itemize}

\item[4.] Return ``0''. $\Box$

\end{itemize}

}
\end{alg}

\begin{rem}
{\rm
The points of the configuration $\mathcal{B}_{\cf}$ are used in the last steps (from 3.4 to 3.6) of the previous algorithm because, there, it is required  to compute and use global sections of sheaves on $\gp^2$ obtained by pushing forward invertible sheaves on $Z_{\cf}$. To perform the remaining steps the algorithm only requires the following data: the degree of $\cf$, the genus $g$ of a general invariant curve, the proximity relations among the points of the configuration ${\mathcal B}_{\cf}$, the above defined numbers $\nu_p(\cf)$ and $\epsilon_p(\cf)$ for each point $p\in {\mathcal B}_{\cf}$ and (only when dic$(\cf)\geq 3$) the degrees of the curves in $S$ and their multiplicities at the points of ${\mathcal B}_{\cf}$.

}
\end{rem}

To end this section  {\it we shall  prove clause a) of Theorems
\ref{teor1} and \ref{teor2}}. In both cases, this clause is an easy consequence of Algorithm \ref{alg1}. Indeed, if  $G$ is a bound on the genus of the rational first integral of a
foliation $\cf$ (assuming that it is different from $1$), then the
degree of the first integral can be bounded by the maximum of the
numbers $2\left(G-1\right)d/\left(\sum_{i=1}^n m_i-3d\right)$
%$$\frac{2(G-1)d}{\sum_{i=1}^n m_i-3d}$$
corresponding to coordinates of primitive lattice representatives
$(d,-m_1,-m_2,\ldots,-m_n)$ of the elements in ${\mathcal R}_{\cf}(S)$
determined by a [dic$(\cf) -2$]-set $S$ of algebraic solutions (notice that, by Lemma \ref{lema1}, there are, at most, two possibilities for these elements). To compute this bound, one needs the following data: the bound $G$, the degree of $\cf$,  the degrees of the curves in $S$ and their multiplicities at the points of ${\mathcal B}_{\cf}$ (only when dic$(\cf)\geq 3$), the proximity relations among the points of the configuration ${\mathcal B}_{\cf}$ and the above defined numbers $\nu_p(\cf)$ and $\epsilon_p(\cf)$ for each point $p\in {\mathcal B}_{\cf}$. Since the last two data only depend on the local analytic type of the dicritical singularities of $\cf$, we conclude clause a) of the mentioned theorems. $\Box$

\section{Foliations with only one dicritical divisor}

\label{sec4}

In this section we shall prove the main result of this paper (Theorem \ref{teor3}), which  for foliations $\cf$ on $\gp^2$ such that dic$(\cf)=1$ solves the Poincaré problem and gives an algorithm to decide algebraic integrability. First we shall show the algorithm (which proves clause b)) and, then, we shall deduce clause a) as a consequence of the results used to justify it.

 For a start, fix a foliation $\cf$ of degree $r$, having a rational first integral and such that dic$(\cf)=1$. To avoid trivialities, we also assume that the cardinality of $\cb_{\cf}$ is greater than 1 (note that otherwise the foliation is defined by a pencil of lines). The following results will allow us state the mentioned algorithm which, as we shall see,  consists, in fact, of two algorithms that must be applied consecutively.

\begin{lem}\label{mou}
Let $\cf$ be as above. All the curves in  the irreducible pencil  $\cp_{\cf}$ defined in Section \ref{sec2} are irreducible and, at most two of them, are non-reduced.
\end{lem}

\begin{proof}
From \cite[Corollary 2]{kaliman} and the subsequent remark, it can be deduced that the cardinality of the set of dicritical exceptional divisors dic$(\cf)$ attached to a foliation $\cf$ satisfies the following inequality
$$
1 +  \sum (e_R-1) \leq \mathrm{dic}(\cf),
$$
where the sum is taken over the set of curves $R$ in the pencil
$\cp_{\cf}$ and $e_R$ stands for the number of different integral
components of $R$. As a consequence any curve in ${\mathcal P}_{\cf}$ is irreducible because dic$(\cf)=1$.
The second part of the statement follows from a result of Poincar\'{e} in \cite[page 187 of I]{poi2}.
%(\cite[Proposition 3.1]{zam1}?).

\end{proof}

\begin{pro}\label{mou2}

Let $\cf$ and ${\mathcal P}_{\cf}$ be as in Lemma \ref{mou}. Let $\mathcal A$ be the set of integral components of the non-reduced curves in ${\mathcal P}_{\cf}$. Then $\deg(A)<\deg(\cf)+2$ for all $A\in {\mathcal A}$. Moreover, if $\mathcal A$ has two elements (say $A_1$ and $A_2$) then $\deg(A_1)+\deg(A_2)=\deg(\cf)+2$.

%Then, the degree of every integral component of a non-reduced curve in ${\mathcal P}_{\cf}$ is less than $\deg(\cf)+2$.
\end{pro}

\begin{proof}

Let $\delta$ be the degree of a primitive rational first integral of
$\cf$ and let $r=\deg(\cf)$. Set $\theta$ the number of non-reduced curves in the pencil
$\cp_{\cf}$ and  $\chi$ the sum of the degrees of the integral
components of these curves. Notice that $\theta\leq 2$ by
Lemma \ref{mou}.

Taking, in $(\ref{equality})$,
intersection products with the total transform of a general line of
$\gp^2$, one has that
\begin{equation}\label{prrr}
2 \delta -r-2=\sum (e_R-1)\deg(R),
\end{equation}
where the sum is taken over the set of integral components $R$ of
the curves in $\cp_{\cf}$ and $e_R$ denotes the multiplicity of $R$
as a component of such curves. Therefore
$$2 \delta-r-2=\theta  \delta - \chi.$$

This concludes the proof of the first assertion because $\theta=1$
implies $\delta < r +2$ and $\theta=2$ shows  $\chi=r+2$, and in
both cases $r+2$ is a strict upper bound of the of the degrees of
the mentioned integral components. The last assertion holds because $\theta=2$ and it is the equality $\chi = r + 2$.

\end{proof}

Next we shall define, for \emph{an arbitrary foliation} $\cf$ (which needs not to have a rational first integral), a set of divisors on $Z_{\cf}$ that will be useful to state our last result before giving the algorithms
that will prove Theorem \ref{teor3}. Let $x$ be a
positive integer and denote by $\Gamma(x)$ the (finite)
set of divisors $C=xL^*-\sum_{p\in {\mathcal
B}_{\cf}} y_p E_p^*$ satisfying the following
conditions:

\begin{itemize}

\item[(a)]   $0 \leq y_p\leq x$ for all $p\in {\mathcal
B}_{\cf}$.

\item[(b)] $C\cdot \tilde{E}_p\geq 0$ for all $p\in {\mathcal
B}_{\cf}$.

\item[(c)] Either $C^2=K_{Z_{\cf}}\cdot C=-1$, or $C^2\leq 0$,
$K_{Z_{\cf}}\cdot C\geq 0$ and $C^2+K_{Z_{\cf}}\cdot C\geq -2$.

\item[(d)] The complete linear system $|C|$ has (projective)
dimension $0$.

\end{itemize}

\begin{lem}\label{zzzz}
Let $\cf$ be as in Lemma \ref{mou}. Any integral component of a non-reduced curve in $\cp_\cf$ is the push-forward ${\pi_{\cf}}_*|C|$ for some divisor  $C$ on $Z_{\cf}$ which belongs to $\bigcup_{x<r+2}\Gamma(x)$.

\end{lem}

\begin{proof}

Let $H$ be an integral component of a non-reduced curve of the
pencil $\cp_\cf$. Let $x$ be the degree of $H$ and, for each $p\in
\cb_{\cf}$, denote by $y_p$ the multiplicity at $p$ of the strict
transform of $H$ on the surface to which $p$ belongs. Then, it holds
the following linear equivalence between divisors on $Z_{\cf}$:
$$\tilde{H}\sim C:=xL^*-\sum_{p\in \cb_{\cf}} y_p E_p^*,$$
and it happens that $H={\pi_{\cf}}_*|C|$. Let us see that $C$
belongs to $\Gamma(x)$. Indeed, condition (a) of the definition of
$\Gamma(x)$ is clear, (b) is true because $\tilde{H}$ is irreducible
and non-exceptional, (c) follows from statement (c) in Proposition
\ref{proposition1} and the Adjunction formula and (d) holds because
the integral components of the curves in ${\pi_{\cf}}_*|C|$ are also
integral components of the curves in the pencil $\cp_{\cf}$. This
concludes the proof because $x<r+2$ by Proposition \ref{mou2}.
\end{proof}

Now we shall give two algorithms (Algorithm \ref{alg2} and Algorithm \ref{alg3}) that, successively applied, {\it prove part b) of Theorem \ref{teor3}}.

Firstly we shall describe the ideas supporting Algorithm \ref{alg2}.
Consider an arbitrary foliation $\cf$ on $\gp^2$  such that
dic$(\cf)=1$ and consider the set $\calr_{\cf}(\emptyset)$. By
Remark \ref{trajesparatodos}, this set must have cardinality $\leq
1$. When  $\calr_{\cf}(\emptyset)=\emptyset$ we can ensure that
$\cf$ has no rational first integral. Otherwise, it is
straightforward to obtain the unique candidate $T$ for the divisor
$T_{\cf}$ defined before Lemma \ref{lema1} (see Remark
\ref{trajesparatodos}). If $\cf$ has a rational first integral,
then, by Lemma \ref{mou}, there exist, at most, two integral
components of non-reduced curves in ${\mathcal P}_{\cf}$. Moreover,
by Lemma \ref{zzzz}, the classes in $A(Z_{\cf})$ of their strict
transforms on $Z_{\cf}$ must be in the finite set
$\bigcup_{x<\deg(\cf)+2}\Gamma(x)$   and furthermore by Proposition
\ref{proposition1}, they must be orthogonal to $T$. Taking into
account these considerations, it can be computed a set $\mathcal A$
such that, if $\cf $ had rational first integral, its elements would
be exactly the integral components of the non-reduced curves in
${\mathcal P}_{\cf}$. The precise algorithm is the following one:

\begin{alg}\label{alg2}

$\;$  \newline \newline \noindent {\it Input:} {\rm A projective differential 1-form
$\mathbf{\Omega}$ defining a foliation $\cf$ on $\gp^2$ of degree $r$ such that dic$(\cf)=1$ and the configuration $\cb_{\cf}$.\\

\noindent {\it Output:} Either a pair $(T, {\mathcal A})$, where $T$
is a candidate for the divisor $T_{\cf}$ and ${\mathcal A}$
is a candidate set for the set of integral components of the
non-reduced curves of $\cp_{\cf}$ (in case $\cf$ being algebraically integrable), or ``0'' (which implies that $\cf$ is
not algebraically integrable).\\

\begin{itemize}

\item[1.] Compute the set $\calr_{\cf}(\emptyset)$. If
$\calr_{\cf}(\emptyset)=\emptyset$ then return ``0''.

\item[2.] Take the unique element $\ell\in \calr_{\cf}(\emptyset)$.

\item[3.] Set $(d,-m_1,-m_2,\ldots,-m_n)$ coordinates in the basis $\mathbf{B}$ of the primitive lattice representative of $\ell$. If $m_i<0$ for some $i$, then return ``0''. Else, let $T:=dL^*-\sum_{i=1}^n m_iE_i^*$.

%\item[5.] If $T\cdot \tilde{A}\not=0$ for some $A\in {\mathcal A}$ then return ``0''.

\item[4.] Seek $C\in \bigcup_{x<\deg(\cf)+2}\Gamma(x)$
such that $T\cdot C=0$ and ${\pi_{\cf}}_*|C|$ is an invariant curve.
Perform this in the following manner: order the divisors in
$[T]^\bot \cap \left(\bigcup_{x<\deg(\cf)+2}\Gamma(x)\right)$  in a
list $C_1,C_2,\ldots$ satisfying  the  following implication:
$$i<j\Rightarrow C_i\in \Gamma(x_i) \mbox{ and } C_j\in \Gamma(x_j)
\mbox{ with } x_i\leq x_j;$$ then, set $C:=C_{i_0}$, $i_0$ being
the minimum index $i$ such that ${\pi_{\cf}}_*|C_i|$ is an invariant
by $\cf$ curve.

\item[5.] If such a divisor does not exist, then return $(T,\emptyset)$. Else,

\begin{itemize}

\item[5.1.] Set $x$ such that $C\in \Gamma(x)$.

\item[5.2.] If $x>\frac{r+2}{2}$, then return $(T,\{C\})$. Else,

\begin{itemize}

\item[5.2.1.] Find  $W\in [T]^\bot \cap \Gamma(r+2-x)$ such that the curve ${\pi_{\cf}}_*|W|$ is invariant and it has not $C$ as integral component.

\item[5.2.2.] If such a curve $W$ does not exist, then return $(T,\{C\})$. Else return $(T,\{C,W\})$.

\end{itemize}

\end{itemize}

\end{itemize}

}
\end{alg}

Finally, we are going to give Algorithm \ref{alg3}, which combined
with Algorithm \ref{alg2} gives rise to our above mentioned
algorithm to decide algebraic integrability. Its inputs will be a
differential 1-form $\mathbf{\Omega}$ defining $\cf$ and the candidate pair
provided by Algorithm \ref{alg2}; the output will be either a
primitive rational first integral for $\cf$ or ``0'' (which means
that $\cf$ is not algebraically integrable). Algorithm \ref{alg3} is
supported on the following result:

%%%%%%%%%%%%%%%%%%%%

\begin{pro}\label{patricio}
\label{U} Let $\cf$ be a foliation on $\gp^2$ such that dic$(\cf)=1$
and it has a rational first integral $f$. Let $\widetilde{D}_f$ be a
general fiber of $\widetilde{f}= f \circ \pi_{\cf}$ and $\gamma$
the positive integer such that $[\widetilde{D}_f] = \gamma
[T_{\cf}]$, $[T_{\cf}]$ being the primitive lattice representative
of $\mathbb{Q}[\widetilde{D}_f]$. Let ${\mathcal A}$ be the set of
integral components of the non-reduced curves in ${\mathcal
P}_{\cf}$. Then, the following statements hold:
\begin{itemize}
\item[(a)] If ${\mathcal A}=\emptyset$, then $\gamma=\frac{r+2}{2s_0}$, $s_0$ being the first coordinate of the class $[T_{\cf}]$ in the basis $\mathbf{B}$.
  \item[(b)] If ${\mathcal A}=\{A_1\}$, then $\gamma=\frac{r+2-\deg(A_1)}{s_0}$.
  \item[(c)] Otherwise ${\mathcal A}=\{A_1,A_2\}$ and then $\gamma=\frac{{\rm lcm}(\deg(A_1),\deg(A_2))}{s_0}$.\end{itemize}
\end{pro}

\begin{proof}

(a) and (b) are direct consequences of equality (\ref{prrr}). To
prove (c) observe first that, by Lemma \ref{mou}, there exist
positive integers $n_1,n_2$ such that the pencil ${\mathcal
P}_{\cf}$ is spanned by homogeneous polynomials giving equations of
$n_1A_1$ and $n_2A_2$. Moreover $n_1$ and $n_2$ are relatively
primes because the pencil is irreducible. Since
$n_1\deg(A_1)=n_2\deg(A_2)$ we have that
$n_1=\frac{\deg(A_2)}{\gcd(\deg(A_1),\deg(A_2))}$  and, therefore,
the degree of a general integral invariant curve is
$$\frac{\deg(A_1)\deg(A_2)}{\gcd(\deg(A_1),\deg(A_2))}={\rm
lcm}(\deg(A_1),\deg(A_2)).$$

\end{proof}

%%%%%%%%%%%%%%%%%%%%

\begin{alg}\label{alg3}

$\;$ \newline \newline \noindent {\it Input:} {\rm A projective differential 1-form
$\mathbf{\Omega}$ defining a foliation $\cf$ on $\gp^2$ such that dic$(\cf)=1$ and a candidate pair $(T, {\mathcal A})$ given by the output of Algorithm \ref{alg2}.\\

\noindent {\it Output:} Either a rational first integral of $\cf$ or ``0''(which means that $\cf$ has no such a first integral).

\begin{itemize}

\item[1.] Compute $\gamma$ according with the values provided in Proposition \ref{U}  (taking the set $\mathcal A$ and the divisor $T$ instead of $T_{\cf}$). If either $\gamma$ is not an integer or  $\gamma T$ is not a divisor then return ``0''.

\item[2.] Compute a basis of the space of global sections $H^0(\gp^2,\pi_{\cf*}\co_{Z_{\cf}}(\gamma T))$.

\item[3.] If $h^0(\gp^2,\pi_{\cf*}\co_{Z_{\cf}}(\gamma T)) \neq 2$, then return ``0''.

\item[4.] Else, take a basis $\{F,G\}$ and check the equality $$\mathbf{\Omega} \wedge (GdF-FdG)=0.$$ If it is satisfied, then  the rational map
$\gp^2\cdots \rightarrow \gp^1$ whose components are $F$ and $G$ is
a primitive rational first integral of $\cf$; return it. Else, return ``0''.

\end{itemize}

}
\end{alg}

Notice that to run Algorithms \ref{alg2} and \ref{alg3}, one only
performs very simple integer arithmetics and resolution of systems
of linear equations.\\

We finish this section by giving the {\it proof of clause a) of
Theorem \ref{teor3}}. That is, we are going to prove the inequality
$$d\leq \frac{(r+2)^2}{4},$$
where $r$ is the degree of a foliation $\cf$ on $\gp^2$ having a
rational first integral of degree $d$ and such that dic$(\cf)=1$.
Let $\mathcal A$ be as in Proposition \ref{patricio}. If either this
set is empty or its cardinality is $1$, the above inequality is
trivially satisfied by clauses (a) and (b) of Proposition
\ref{patricio}. Therefore, let us take ${\mathcal A}=\{A_1,A_2\}$.
Applying Proposition \ref{mou2} and clause (c) of Proposition
\ref{patricio} one has that $$d={\rm
lcm}(\deg(A_1),r+2-\deg(A_1))\leq \deg(A_1)(r+2-\deg(A_1))\leq
\frac{(r+2)^2}{4},$$ completing the proof of Theorem \ref{teor3}.
$\Box$

\begin{rem}
{\rm Let $\cf$ be a foliation on $\gp^2$ such that dic$(\cf)=1$ and
the coefficients of a  differential 1-form $\mathbf{\Omega}$
providing $\cf$ are integer numbers. Then, since we have a bound on
the degree of the first integral (if it exists), an alternative
algorithm to compute that integral is that described in \cite{cheze}
which relies on the factorization of the extactic curves studied in
\cite{per} (see also \cite{do-lo}). Nevertheless, to check whether
dic$(\cf)=1$, a resolution of the singularities of $\cf$ is needed.}
\end{rem}

\section{Examples}
\label{sec5}

This last section is devoted to provide some examples that show how
our algorithms work. For a start, we shall use Algorithm \ref{alg1}
to get a rational first integral of a foliation of degree 4.

\begin{exa}\label{ex1}
{\rm Consider the singular algebraic foliation $\cf$  given by the
differential 1-form
$$\mathbf{\Omega} = (2X_1 X_2^5)\; dX_0 + (-7X_1^{5} X_2 - 3X_0 X_2^5+X_1 X_2^5) \;dX_1 + (7X_1^{6} + X_0 X_1 X_2^4-X_1^2
X_2^4) \; dX_2.$$

From the minimal resolution of singularities we compute the
configuration of dicritical points $\cb_\cf$. It has 13 points,
$\{p_i\}_{i=1}^{13}$, and its proximity graph is displayed in Figure
1. We recall that the vertices of this graph represent the points in
$\cb_\cf$, and two vertices, $p_i,p_j \in \cb_\cf$, are joined by an
edge if $p_i$ belongs to the strict transform of the exceptional
divisor $E_{p_j}$. This edge is curved-dotted except when $p_i$
belongs to the first infinitesimal neighborhood of $p_j$ (here the
edge is straight-continuous). For simplicity's sake, we delete those
edges which can be deduced from others (for instance, we have
deleted the curved-dotted edges joining $p_7$ and $p_6$ with $p_4$
since there is an edge joining $p_8$ and $p_4$). From the local
differential 1-forms defining the transformed foliations in the
resolution process, it can be easily deduced that the unique
dicritical divisors are $E_{p_3}$ and $E_{p_{13}}$· Therefore
dic$(\cf) =2$. Then we can use Algorithm \ref{alg1} to check whether
$\cf$ has a rational first integral of genus $g=0$ because $g \neq
1$ and $\cf$ admits an empty $[\mathrm{dic}(\cf)-2]$-set of
independent algebraic solutions. From the minimal resolution, we can
obtain the divisor class
$$[K_{\tilde{\cf}}-K_{Z_{\cf}}]=7[L^*] -[E_{p_1}^*] -[E_{p_2}^*]-2[E_{p_3}^*] -5[E_{p_4}^*]-2\sum_{i=5}^8[E_{p_i}^*]-\sum_{i=9}^{12}[E_{p_i}^*]-2[E_{p_{13}}^*]$$
and the set  $\mathcal{R}_\cf (\emptyset)$:
$$\mathcal{R}_\cf (\emptyset)=\{(10:-2:-1:-1:-8:-2:-2:-2:-2:-2:-2:-2:-1:-1),$$
$$(2770: -762: -381: -381: -2152: -538:$$ $$ -538: -538: -538:-538: -538: -538: -269: -269)\},$$
where we have taken projective coordinates with respect to the basis
$\mathbf B$. Following Algorithm \ref{alg1}, we must consider the
first  element in $\mathcal{R}_\cf (\emptyset)$ and compute the
value $\alpha$ in step 3.3. Here $\alpha=1$,
$$T = 10L^*-2E_{p_1}^*-E_{p_2}^*-E_{p_3}^*-8E_{p_4}^*-2\sum_{i=5}^{11}
E_{p_i}^*-E_{p_{12}}^*-E_{p_{13}}^*$$ and the dimension of the
vector space $H^0(\gp^2,\pi_{\cf*}\co_{Z_{\cf}}( T))$ is two, being
$F =X_1^3 X_2^7$ and
$G=X_1^{10}-2X_0X_1^5X_2^4+2X_1^6X_2^4+X_0^2X_2^8-2X_0 X_1
X_2^8+X_1^2 X_2^8$ a basis of this vector space. Finally, $\cf$ has
a rational first integral given by $F$ and $G$ because $\mathbf{\Omega}
\wedge (GdF-FdG)=0$. Notice that this example is \cite[Example
2]{g-m-1}, where we proved the same result with a different
procedure. }
\end{exa}

\begin{figure}[hbt]\label{fig1}
\setlength{\unitlength}{1mm}
\begin{center}
\begin{picture}(40,45)

\put(0,0){\circle*{2}} \put(0,0){\line(0,1){5}} \put(5,0){$p_1$}

\put(0,5){\circle*{2}}\put(0,5){\line(0,1){5}} \put(5,5){$p_2$}

\put(0,10){\circle*{2}} \put(5,10){$p_3$}

\put(30,0){\circle*{2}} \put(30,0){\line(0,1){5}} \put(35,0){$p_4$}

\put(30,5){\circle*{2}} \put(30,5){\line(0,1){5}} \put(35,5){$p_5$}

\put(30,10){\circle*{2}} \put(30,10){\line(0,1){5}}
\put(35,10){$p_6$}

\put(30,15){\circle*{2}} \put(30,15){\line(0,1){5}}
\put(35,15){$p_7$}

\put(30,20){\circle*{2}} \put(30,20){\line(0,1){5}}
\put(35,20){$p_8$}

\put(30,25){\circle*{2}} \put(30,25){\line(0,1){5}}
\put(35,25){$p_9$}

\put(30,30){\circle*{2}} \put(30,30){\line(0,1){5}}
\put(35,30){$p_{10}$}

\put(30,35){\circle*{2}} \put(30,35){\line(0,1){5}}
\put(35,35){$p_{11}$}

\put(30,40){\circle*{2}} \put(30,40){\line(0,1){5}}
\put(35,40){$p_{12}$}

\put(30,45){\circle*{2}} \put(35,45){$p_{13}$}

\qbezier[30](30,0)(15,10)(30,20)

\qbezier[20](30,35)(20,40)(30,45)

\qbezier[20](0,0)(-10,5)(0,10)

\end{picture}
\end{center}
\caption{Proximity graph of $\cb_\cf$ in Example \ref{ex1}
}
\end{figure}
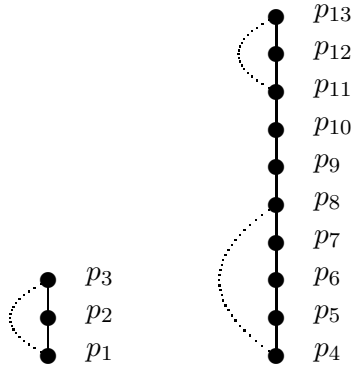

\begin{exa}\label{ex2}
{\rm Set $\cf$ the foliation attached to the differential 1-form
$$
\mathbf{\Omega} = (3X_0^2 X_2^3) \;dX_0 - (5X_1^4X_2)\; dX_1 + (5 X_1^5 -3 X_0^3 X_2^2)\; dX_2.
$$
The configuration ${\mathcal B}_{\cf}$ has 19 points $\{p_i\}_{i=1}^{19}$ and only one dicritical divisor: $E_{p_{19}}$. We show the corresponding proximity graph in Figure 2. From the resolution of singularities it is deduced that
$$[K_{\tilde{\cf}}-K_{Z_{\cf}}]=6[L^*] -2[E_{p_1}^*] -3[E_{p_2}^*]-2[E_{p_3}^*] -2[E_{p_4}^*]-\sum_{i=5}^{18}[E_{p_i}^*]-2[E_{p_{19}}^*]$$
and, moreover, it can be checked that
$$(5: -2: -2: -1: -1: \cdots: -1)$$
are the projective coordinates with respect to the basis $\mathbf B$ of the unique element of the set $\mathcal{R}_\cf (\emptyset)$. Applying Algorithm \ref{alg2} we get the pair $(T,{\mathcal A})$,  $T$ being the divisor $5L^*-2E_{p_1}^*-2E_{p_2}^*-\sum_{i=3}^{19} E_{p_i}^*$ and ${\mathcal A}=\{(X_2 =0)\}$. Now, applying Algorithm \ref{alg3} we compute $\gamma = \frac{r+2-\deg(A_1)}{s_0}=\frac{4+2-1}{5}=1$ (step 1). Performing the remaining steps of the algorithm we conclude that $\cf$ admits a rational first integral given by $F = X_1^5 - X_0^3 X_2^2$ and $G=X_2^5$.
}
\end{exa}

\begin{figure}[hbt]\label{fig2}
\setlength{\unitlength}{1mm}
\begin{center}
\begin{picture}(20,35)

\put(0,0){\circle*{2}} \put(0,0){\line(0,1){5}} \put(5,0){$p_1$}

\put(0,5){\circle*{2}}\put(0,5){\line(0,1){5}} \put(5,5){$p_2$}

\put(0,10){\circle*{2}} \put(0,10){\line(0,1){5}} \put(5,10){$p_3$}

\put(0,15){\circle*{2}} \put(0,15){\line(0,1){5}} \put(5,15){$p_4$}

\put(0,20){\circle*{2}} \put(0,20){\line(0,1){5}} \put(5,20){$p_5$}

\put(0,25){\circle*{2}}
%\put(0,25){\line(0,1){5}}
\put(5,25){$p_6$}

\put(0,30){$\vdots$}

\put(0,38){\circle*{2}} \put(5,38){$p_{19}$}

\qbezier[20](0,5)(-10,10)(0,15)

\end{picture}
\end{center}
\caption{Proximity graph of $\cb_\cf$ in Examples
\ref{ex2} and \ref{ex4}
}
\end{figure}
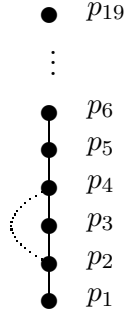

\begin{exa}\label{ex3}
{\rm
Consider now the foliation $\cf$
defined by the projective differential 1-form $\mathbf{\Omega}=AdX_0+BdX_1+CdX_2$, where
$$A=8X_0^4X_1^2+10X_0X_1^5+2X_0^5X_2-4X_0^2X_1^3X_2-4X_0^3X_1X_2^2-4X_1^4X_2^2+2X_0X_1^2X_23,$$
$$B=-8X_0^5X_1-10X_0^2X_1^4+10X_0^3X_1^2X_2+5X_1^5X_2-X_0^4X_2^2-2X_0X_1^3X_2^2+2X_0^2X_1X_2^3-X_1^2X_2^4,$$
$$C= -2X_0^6-6X_0^3X_1^3-5X_1^6+5X_0^4X_1X_0+6X_0X_1^4X_2-4X_0^2X_1^2X_2^2+X_1^3X_2^3.$$
$\cb_\cf=\{p_i\}_{i=1}^{10}$ and the reader can see its proximity
graph in Figure 3. Algorithm \ref{alg2} gives the pair $(T,{\mathcal
A})$, where $$T = 10L^*-4\sum_{i=1}^{6} E_{p_i}^*-\sum_{i=7}^{10}
E_{p_i}^*$$ and ${\mathcal A}=\{(F_1=0), (F_2=0)\}$, being $$F_1 =
X_1 X_2 - X_0^2 \mbox{ and } F_2=
2X_0^3X_1^2+X_1^5+X_0^4X_2-2X_0X_1^3X_2-2X_0^2X_1X_2^2+X_1^2X_2^3.$$
The value $\gamma$ in Algorithm \ref{alg3} is $\gamma={\rm
lcm}(2,5)/10=1$ and a basis of $H^0(\gp^2,\pi_{\cf*}\co_{Z_{\cf}}(
T))$ is $\{F:=F_1^5, G:=F_2^2\}$, which defines a rational first
integral of $\cf$ since $\mathbf{\Omega} \wedge (GdF-FdG) =0$. }
\end{exa}

\begin{figure}[hbt]\label{fig3}
\setlength{\unitlength}{1mm}
\begin{center}
\begin{picture}(30,45)

\put(0,0){\circle*{2}} \put(0,0){\line(0,1){5}} \put(5,0){$p_1$}

\put(0,5){\circle*{2}} \put(0,5){\line(0,1){5}} \put(5,5){$p_2$}

\put(0,10){\circle*{2}} \put(0,10){\line(0,1){5}}
\put(5,10){$p_3$}

\put(0,15){\circle*{2}} \put(0,15){\line(0,1){5}}
\put(5,15){$p_4$}

\put(0,20){\circle*{2}} \put(0,20){\line(0,1){5}}
\put(5,20){$p_5$}

\put(0,25){\circle*{2}} \put(0,25){\line(0,1){5}}
\put(5,25){$p_6$}

\put(0,30){\circle*{2}} \put(0,30){\line(0,1){5}}
\put(5,30){$p_{7}$}

\put(0,35){\circle*{2}} \put(0,35){\line(0,1){5}}
\put(5,35){$p_{8}$}

\put(0,40){\circle*{2}} \put(0,40){\line(0,1){5}}
\put(5,40){$p_{9}$}

\put(0,45){\circle*{2}} \put(5,45){$p_{10}$}

\qbezier[30](0,25)(-10,35)(0,45)

\end{picture}
\end{center}
\caption{The proximity graph of ${\mathcal B}_{\cf}$ in Example
\ref{ex3}}
\end{figure}
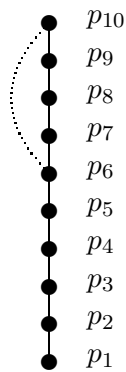

\begin{exa}\label{ex4}
{\rm Let $\cf$ be the foliation given by the differential 1-form
$$
\mathbf{\Omega} = (3X_0^2 X_2^3-X_1^2 X_2^3) \;dX_0 - (5X_1^4X_2-X_0X_1X_2^3)\; dX_1 + (5 X_1^5 -3 X_0^3 X_2^2)\; dX_2.
$$
The configuration ${\mathcal B}_{\cf}$ has 19 points
$\{p_i\}_{i=1}^{19}$, only one dicritical divisor, $E_{p_{19}}$, and
its proximity graph is that of Figure 2. Applying Algorithm
\ref{alg2} one obtains the pair $(T,{\mathcal A})$, where $T$ has
the same expression than the one of Example \ref{ex2} and the unique
element of ${\mathcal A}$ is also the line $(X_2=0)$. Then the value
$\gamma$ given in Algorithm \ref{alg3} is also $\gamma=1$. It can be
checked that $H^0(\gp^2,\pi_{\cf*}\co_{Z_{\cf}}( T))$ is the pencil
generated by $X_1X_2^4$ and $X_2^5$, which is not irreducible and,
therefore,  does not provide a primitive rational first integral of
$\cf$. As a consequence, $\cf$ does not admit a rational first
integral.

}

\end{exa}

\end{document}